\documentclass[11pt]{amsart}
\usepackage{geometry,a4wide,amssymb,amsfonts,amsmath}  

\usepackage[alphabetic]{amsrefs}
\usepackage{bbm,xcolor}
\usepackage{dsfont}
\usepackage{yfonts}
\usepackage[greek, english]{babel}
\usepackage{mathtools}
\usepackage{mathrsfs}
\usepackage{hyperref}  
\usepackage{footmisc} 
\usepackage{comment}

\usepackage[normalem]{ulem}

\usepackage{utfsym}
\usepackage{twemojis}
\usepackage{fontawesome5}
\usepackage{figchild}

\usepackage{amsaddr}

\newcommand{\R}{\mathbb{R}}
\newcommand{\st}{\mathfrak{t}}

\newcommand{\N}{\mathbb{N}}

\newcommand{\Z}{\mathbb{Z}}

\newcommand{\F}{\mathcal{F}}

\newcommand{\M}{\mathcal{M}}
\newcommand{\Lt}{\mathcal{L}}

\newcommand{\var}{\mathtt{var}}
\newcommand{\ou}{\"o}

\newcommand{\1}{\mathds{1}}

\newcommand{\tei}{\mathtt{tei}}
\newcommand{\eie}{\mathtt{eie}}

\newcommand{\cpl}{\mathtt{Cpl}}

\newcommand{\lip}{\mathtt{Lip}}

\theoremstyle{plain}

\newtheorem{thm}{Theorem}[section]

\newtheorem{prop}[thm]{Proposition}
\newtheorem{cor}[thm]{Corollary}

\newtheorem{example}[thm]{Example}
\newtheorem{rmk}[thm]{Remark}

\title[Transport-entropy and EKP inequalities on countable Markov shifts]
{Effective intrinsic ergodicity and the transport-entropy inequality: a relation and applications
}

\author{\small{Mirmuhsin Maxmud}}
\address{\small{Research Unit of Mathematical Sciences, The University of Oulu}}
\email{Mirmukhsin.Makhmudov@oulu.fi}

\begin{document}

\begin{abstract}
Effective intrinsic ergodicity (EIE) quantifies how the pressure defect of an invariant measure controls its distance from an equilibrium state. 
In this paper, we make explicit the connection between the Einsiedler-Kadyrov-Polo (EKP) inequality, which is central to EIE, and the transport–entropy inequality (TEI).
We also extend the finite-state transport-entropy inequality of \cite{CMU19} to topologically mixing countable Markov shifts with the Big Images and Preimages (BIP) property, yielding an EIE estimate in Ornstein’s \(\bar d\)-metric that strengthens \cite{Ruh21}. 
We also establish a local $1/2$-Hölder bound for the map sending a Hölder potential of finite pressure to its Bowen--Gibbs measure, with respect to the uniform norm on potentials and the $\bar d$-metric on measures.
This result extends some aspects of Bhullar's continuity theorem \cite{B25}: it applies to proper subshifts, provides an explicit modulus of continuity with respect to the uniform norm, and does not require strong positive recurrence.

\end{abstract}

\maketitle

\section{Introduction}

In recent years, quantitative estimates on the uniqueness of equilibrium states have attracted growing attention in both ergodic theory and mathematical statistical mechanics. 
Two closely related approaches have developed largely in parallel, with, to the best of the author's knowledge, no cross-referencing despite sharing a common information-theoretic structure.
One purpose of this paper, among other things, is to make the connection between these approaches explicit.

On the ergodic-theoretic side, the notion of \textit{effective intrinsic ergodicity} (EIE) quantifies how the equilibrium state is distinguished from other invariant probability measures by their pressure defects. 
The underlying estimate originates in an argument of Einsiedler, recorded in Polo's thesis \cite{Polo-thesis}, and was developed further by Kadyrov \cite{Kad15} and subsequent authors \cites{Ruh16, Ruh21, RS22, P26, HKMM26, CK26}. 
More precisely, let \(\Omega\subset E^{\mathbb Z_+}\), where \(E\) is at most countable and is equipped with the discrete topology, be a topologically mixing subshift, and let \(S:\Omega\to\Omega\) denote the left shift. 
Suppose that \(\phi:\Omega\to\mathbb R\) is H\ou lder with finite pressure \(P(\phi)\), is bounded from above, and that \(\mu_\phi\) is its unique equilibrium state (see \eqref{eq: VP in incorporated form for cms}). 
We say that \(\mu_\phi\) is \textit{effectively intrinsically ergodic} if there exists a constant \(C_{\eie}>0\) such that, for every \(S\)-invariant probability measure \(\nu\in\mathcal M_{1,S}(\Omega)\) 
and every bounded H\"older function \(f:\Omega\to\mathbb R\),
\begin{equation}
\label{eq: ekp in its original form}
    \left|\int_\Omega f\,d\nu-\int_\Omega f\,d\mu_\phi\right|
    \leq C_{\eie}\,\|f\|_\theta\,
    \sqrt{P(\phi)-h(\nu)-\int_\Omega\phi\,d\nu},
\end{equation}
with the convention that if $(h(\nu),\int_\Omega \phi d\nu)= (+\infty,-\infty)$, then 
$h(\nu)+\int_\Omega \phi d\nu$ stands for 
$\int_\Omega (I_\nu+\phi)d\nu$,
where $\|f\|_\theta:=\|f\|_\infty+|f|_\theta$, $|f|_\theta$ is the Lipschitz seminorm with respect to the symbolic metric \(d_\theta\), \(h(\nu)\) denotes the measure-theoretic entropy of \(\nu\), and $I_\nu(\omega)=-\log\nu(\omega_0|\omega_1^\infty), \; \omega\in\Omega$ is the information function of $\nu$, for which $\int_\Omega I_\nu d\nu=h(\nu)$; see Proposition \ref{prop: rel ent den=press gap}. 
The estimate in \eqref{eq: ekp in its original form} is commonly referred to as the \textit{Einsiedler--Kadyrov--Polo} (EKP) inequality. 
By the Kantorovich--Rubinstein duality theorem \cite{Dudley-book}*{Theorem 11.8.2}, \eqref{eq: ekp in its original form} can also be written in a different (slightly stronger) form using the Wasserstein-1 metric $W_{1,d_\theta}$ defined on the space of probability measures $\M_1(\Omega)$: for every $\nu\in\mathcal M_{1,S}(\Omega)$,
\begin{equation}\label{eq: ekp in W_{1,d_theta} norm}
    W_{1,d_\theta}(\nu,\mu_\phi)\leq C_{\eie} 
    \sqrt{P(\phi)-h(\nu)-\int_\Omega \phi d\nu}.
\end{equation}

In parallel, the transport-entropy inequality alongside Gaussian concentration bounds and estimates in Ornstein's $\bar d$-metric \cites{Wu06, CMU19, CCR26} have been developed in the framework of statistical mechanics and probability.
In this setting, a 
shift-invariant probability measure $\mu$ on a subshift $\Omega\subset E^{\Z_+}$ is said to satisfy the \textit{\textbf{transport-entropy inequality}} (TEI) with a constant $C_{\tei}>0$ if for all shift-invariant measure $\nu\in\M_{1,S}(\Omega)$, the \textit{\textbf{relative entropy density}} $h(\nu|\mu)$ exists and
\begin{equation}\label{eq: tei}
    \bar d(\nu,\mu)\leq C_\tei \sqrt{h(\nu|\mu)}.
\end{equation}
We note that the relative entropy density is the infinite-volume analogue of relative entropy, and plays an important role in statistical mechanics.
In infinite-volume systems, one often needs to compare two mutually singular probability measures, such as those describing distinct phases of the same model \cite{Georgii-book}. 
In such cases, the ordinary relative entropy is infinite and therefore provides no finite measure of the discrepancy, whereas the relative entropy density may remain finite.
We define the relative entropy density in Section \ref{section: TEI implies EKP} and the Ornstein $\bar d$-distance in Section \ref{preliminaries}.

The connection between the TEI \eqref{eq: tei} and the EKP inequality \eqref{eq: ekp in its original form} becomes transparent via the following identity, which expresses the relative entropy density as the pressure defect:
\begin{equation}\label{eq: rel entropy=pressure in intro}
    h(\nu|\mu_\phi)
    =
    P(\phi)-\int_\Omega (I_\nu+\phi) d\nu,\;\; \nu\in \M_{1,S}(\Omega).
\end{equation}
The relation \eqref{eq: rel entropy=pressure in intro} holds true in very general setups, in particular, when the shift space $\Omega$ has the BIP property, introduced in Section \ref{preliminaries}, and the potential $\phi$ is H\ou lder with finite pressure; see Proposition \ref{prop: rel ent den=press gap}.

When $(\Omega,S)$ is a finite-state full shift, \eqref{eq: rel entropy=pressure in intro} has been proven for all -- possible multiple -- equilibrium states of an \textit{extensible potential} in \cite{MV2026}. 
This relation was first established by H. F\ou llmer in 1973 for finite-range interactions on $\mathbb Z^d$ with a finite state space \cite{Foll73}*{Eq. 4.25}, and has since been extended to various settings in statistical mechanics; see the bibliographical notes in \cite{Georgii-book}.
In ergodic theory, by contrast, the variational principle is more commonly expressed in its optimization form,
$$
P(\phi)=
\sup_{\nu\in\M_{1,S}(\Omega)}
\Big\{h(\nu)+\int_\Omega\phi \; d\nu\Big\}.
$$
The pressure-defect identity \eqref{eq: rel entropy=pressure in intro} makes explicit how the failure of an invariant measure to attain this supremum is quantified by its relative entropy density with respect to the equilibrium measure $\mu_\phi$.

Recently, Climenhaga and Kumari \cite{CK26} studied effective intrinsic ergodicity on topologically mixing subshifts of finite type (SFTs), establishing a quantitative estimate in Ornstein's $\bar d$-metric for the unique measure of maximal entropy, which is the equilibrium state of the zero potential. Their approach draws on Marton's concentration inequalities \cite{Mar96}. In the terminology of the present paper, their main result, Theorem A, can be expressed as a transport-entropy inequality for the measure of maximal entropy, with the relative entropy density given by the corresponding entropy defect through \eqref{eq: rel entropy=pressure in intro}. 
This TEI is also covered by the more general result of \cite{CMU19} (Theorem 3.9) whose applicability to topologically mixing SFTs is explicitly noted by its authors. 
This provides a concrete instance of the parallel development of the two approaches discussed above, and illustrates the value of making their relationship explicit.
Climenhaga and Kumari also present their result as strengthening Kadyrov's earlier estimate \cite{Kad15} by obtaining a corresponding estimate in Ornstein's $\bar d$-metric instead of the Wasserstein-$1$ metric.
The relationship between these two metrics is therefore relevant to understanding this strengthening. In Proposition \ref{prop: Was dist vs bar d dist}, we establish an explicit comparison between them, thereby clarifying the connection between Wasserstein estimates and estimates in the $\bar d$-metric.

In the present paper, we also extend the transport-entropy inequality of \cite{CMU19}*{Theorem 3.9} to the setup of topologically mixing countable Markov shifts with the BIP property; see Theorem \ref{thm: my tei on CMS}.
The proof of Theorem \ref{thm: my tei on CMS} relies on the Gaussian concentration bound established in \cite{MMNL21}*{Theorem 3.1} and the Bobkov-G\"otze theorem \cite{BG99}*{Theorem 3.1}. 
We then apply Theorem \ref{thm: my tei on CMS} to two questions in ergodic theory.

Our first application is to the effective intrinsic ergodicity of the equilibrium states for the H\"older potentials on the countable Markov shifts.
By combining Theorem \ref{thm: my tei on CMS} and Proposition \ref{prop: rel ent den=press gap}, we strengthen the main result of \cite{Ruh21} by obtaining a corresponding bound in Ornstein's $\bar d$ metric; see Corollary \ref{cor: EKP in bar d on CMS, recovery of Ruhr21}. 
This result also extends the effective intrinsic ergodicity estimate of \cite{CK26} to the countable Markov shifts.
We also note that the setting of Corollary \ref{cor: EKP in bar d on CMS, recovery of Ruhr21} differs substantially from that of \cite{RS22}. The former assumes that the shift space $(\Omega,S)$ has the BIP property, whereas the latter assumes finite Gurevich entropy. For a topologically mixing countable Markov shift, these two conditions hold simultaneously if and only if the state space $E$ is finite. Thus, in the infinite-state case, our result and those of \cite{RS22} address disjoint regimes.

Our second application concerns the continuity, in the $\bar d$ topology, of the map $\phi\mapsto\mu_\phi$ that sends a
Hölder potential $\phi$ to its equilibrium state $\phi$, a question raised in \cite{CQ98}.
Recently, Bhullar showed that this map is continuous in the $\bar d$ topology when the countable Markov shift $\Omega$ is full, i.e., $\Omega=E^{\Z_+}$ \cite{B25}*{Theorem 1.1}.
By Theorem \ref{thm: my tei on CMS}, we obtain a relevant quantitative estimate. 
If $\Omega$ be topologically mixing countable Markov shift with the BIP property, then for every H\ou lder potential $\phi$ with finite pressure, there exists a constant $C(\phi)$ such that for arbitrary H\"older potential $\psi$ on $\Omega$ with finite pressure,
\begin{equation}\label{eq: Holderianity of Coelho-Quas's map intro}
    \bar d(\mu_\phi,\mu_\psi)
    \leq
    C(\phi) \lVert \phi- \psi\rVert_\infty^{\frac 12},
\end{equation}
where $\mu_\phi$ and $\mu_\psi$ are the corresponding Bowen-Gibbs measures; see Corollary \ref{assertion: Holder continuity of Coelho-Quas map}.
This estimate improves the results of \cites{CQ98, B25} in two directions:
\begin{itemize}
    \item[(1)] it applies to the proper subshifts of $E^{\Z_+}$ not only the full shift $E^{\Z_+}$;
    
    \item[(2)] it provides
    an explicit modulus of continuity with respect to the uniform metric on the space of potentials.
    This distinguishes it from the continuity results of Bhullar \cite{B25} and Coelho and Quas \cite{CQ98}, which employ metrics induced by norms that control the variations of the potential in addition to its uniform norm. 
    Thus, \eqref{eq: Holderianity of Coelho-Quas's map intro} establishes continuity with respect to a weaker notion of closeness between potentials.
\end{itemize}

The paper is organised as follows: in Section \ref{preliminaries}, we introduce the necessary notions for the arguments of this paper; 
in Section \ref{section: TEI implies EKP}, we make explicit the relation between the EKP and transport-entropy inequalities through Proposition \ref{prop: rel ent den=press gap} and Proposition \ref{prop: Was dist vs bar d dist}; 
Section \ref{Section: TEI on CMS}, is dedicated to Theorem \ref{thm: my tei on CMS} which establishes the transport-entropy inequality on countable Markov shifts and to its proof and its consequences such as Corollary \ref{cor: EKP in bar d on CMS, recovery of Ruhr21} and Corollary \ref{assertion: Holder continuity of Coelho-Quas map}; and the last section is dedicated to the discussion of the scope of transport-entropy inequality and its relations with other concentration inequalities.

\section{Preliminaries 
}\label{preliminaries}

\subsection{Setup}
We consider a finite or countably infinite set $E$ for the state space, $E\times E$ matrix $\mathtt{A}=(\mathtt{A}_{ab})$ of $0$ and $1$'s, and the associated countable Markov shift $\Omega$, which is defined as
$$
\Omega:=\{\omega\in E^{\Z_+}: \mathtt{A}_{\omega_i\omega_{i+1}}=1, \; i\in\Z_+\}.
$$
A finite word $\omega_0\omega_1\dots \omega_{n-1}$ of $n$ symbols of $E$, abbreviated as $\omega_0^{n-1}$, is called admissible if the corresponding cylinder set $[\omega_0^{n-1}]:=\{\xi\in\Omega: \xi_i=\omega_i \text{ for } i=0,\dots n-1\}$ is nonempty.
We denote the set of admissible words of length $n$ by $\Lt_n$.
We equip $\Omega$ with the left-shift map $S:\Omega\to\Omega$, $S(\omega)=(\omega_{i+1})_{i\in\Z_+}, \; \omega\in\Omega$ and the metric $d_\theta(\xi,\eta):=\theta^{\min\{i\in\Z_+:\xi_i\neq \eta_i \}}$ for $\xi,\eta\in\Omega$ with the convention that $d_\theta(\xi,\xi)=0$, where $\theta\in(0,1)$.
Throughout the paper, we assume that the dynamical system $(\Omega, S)$ is \textit{topologically mixing}, i.e., for all $a,b \in E$, there exists a number $N(a,b)$ such that for every $n\geq N(a,b)$, there is an admissible word $\omega_0^{n-1}$ of length $n$ satisfying $\omega_0=a$ and $\omega_{n-1}=b$.
It is said that $(\Omega,S)$ satisfies the \textit{Big Images and Preimages} (BIP) property if there exists a finite set $\{b_1,\dots b_N\}\subset E$ of states such that for every $a\in E$, there exist $i,j$ satisfying $\mathtt{a}_{b_ia}\mathtt{a_{ab_j}}=1$ \cite{s99}.
We note that if $E$ is finite, then $\Omega$ automatically has the BIP property.

For a volume $V\subset \Z_+$, $\F_V$ denotes the smallest $\sigma-$algebra that makes, for every $n\in V$, the projection maps $\omega\in\Omega\mapsto \omega_n\in E$ measurable.
We use $\F$ for $\F_{[0,\infty)}$.
We denote the set of probability measures on $\Omega$ by $\M_{1}(\Omega)$, and the subset of shift-invariant measures by $\M_{1,S}(\Omega)$.
For a potential $\phi:\Omega\to\R$, we denote its $n$-th variation by $\var_n\phi$, i.e., 
$$
\var_n\phi:=\sup\{\phi(\xi)-\phi(\eta): \xi,\eta\in\Omega, \;\; \xi_0^{n-1}=\eta_0^{n-1}\}.
$$
For an ergodic sum $\sum_{i=0}^{n-1}\phi\circ S^{i}$ we use the notation $S_n\phi$.
We call a potential $\phi:\Omega\to\R$ \textit{H\ou lder}, if its Lipschitz constant $|\phi|_{\theta}:=\sup_{n\geq 1}\frac{\var_n\phi}{\theta^n}$ is finite.
We note that when $E$ is infinite, H\ou lder continuity does not imply the boundedness of the potential.

\subsection{The variational principle for the countable Markov shifts}\label{section: vp for cms}
Consider a H\ou lder potential  $\phi:\Omega\to\R$ on a topologically mixing countable Markov shift $\Omega$ satisfying the BIP property.
The (Gurevich) pressure $P(\phi)$ of $\phi$ is defined as the limit
\begin{equation}\label{eq: gurevich pressure}
    P(\phi):=\lim_{n\to\infty}\frac{1}{n}\log \sum_{\omega: S^n\omega=\omega} e^{S_n\phi(\omega)}\1_{[a_0]}(\omega),
\end{equation}
where $a$ is an element of the state space $E$.
In \cite{s99}*{Theorem 1}, it is shown that the above limit exists, which could be $+\infty$ but not $-\infty$, and is independent of the letter $a\in E$.
Throughout the paper, we consider potentials with finite pressure. 
In light of \cite{s99}*{Theorem 3} and \cite{s03}*{Corollary 2}, if the pressure is finite and the shift space $\Omega$ has the BIP property, then the following variational equation holds:
\begin{equation}\label{eq: VP in separated form for cms}
    P(\phi)=\sup\Big \{h(\nu)+\int_\Omega\phi d\nu: \nu\in\M_{1,S}(\Omega),\; \int_\Omega\min\{\phi,0\} d\nu>-\infty\Big\},
\end{equation}
where $h(\nu):=\lim_{n\to\infty} -\frac{1}{n}\sum_{\omega_0^{n-1}\in \Lt_n}\nu([\omega_0^{n-1}])\log \nu([\omega_0^{n-1}])\in[0,\infty]$ is the measure-theoretic entropy of the measure $\nu$.
A measure $\nu\in\M_{1,S}(\Omega)$ that realises the variational equation \eqref{eq: VP in separated form for cms} is called an \textit{equilibrium state} for the potential $\phi$.
In this sense,
in the countable state setting, unlike the finite state setting, an equilibrium state may not exist, even for a H\ou lder potential.

The condition $\int_\Omega\min\{\phi,0\} d\nu>-\infty$ in the variational equation \eqref{eq: VP in separated form for cms} is important; otherwise one can end up with uncertainty "$\infty-\infty$" in $h(\nu)+\int_\Omega\phi d\nu$, as in the case in the following folklore example.

\begin{example}\label{the folklore example in cms}
Let $\Omega:=\Z_{+}^{\Z_+}$ and pick a probability mass $\{p_i>0: i\in\Z_+\}$ such that $\sum_{i=0}^\infty p_i\log p_i=-\infty$.
Consider a potential
$
\varphi(\omega)=\log p_{\omega_0}, \; \omega\in\Omega
,$
which is H\ou lder since for every  $n\in\N$, $\var_n\varphi=0$, and also consider the Bernoulli measure $\lambda:=(\sum_{i\in\Z_+} p_i\delta_{\{i\}})^{\Z_+}$ on $\Omega$.
By our choice, $\lambda$ is shift-invariant and $h(\lambda)=+\infty$.
Furthermore, it can be easily checked that $\int_\Omega \varphi d\lambda=-\infty$ and  $P(\varphi)=0$.    
\end{example}
Nevertheless, one can still drop the condition $\int_\Omega\min\{\phi,0\} d\nu>-\infty$ at the cost of incorporating the two terms in $h(\nu)+\int_\Omega\phi d\nu$ into one integral $\int_\Omega(I_\nu+\phi)d\nu$, since $\int_\Omega I_\nu d\nu=h(\nu)$.
After such incorporation, as it is shown in \cite{s99}*{Corollary 2}, we have the following variational equation:
\begin{equation}\label{eq: VP in incorporated form for cms}
    P(\phi)=\sup\Big \{\int_\Omega(I_\nu+\phi)d\nu: \nu\in\M_{1,S}(\Omega)\Big\}.
\end{equation}
For a topologically mixing countable Markov shift $\Omega$ with the BIP property and for a H\ou lder potential $\phi$ with finite pressure, 
the supremum in \eqref{eq: VP in incorporated form for cms} is actually a maximum, namely, there exists a unique measure $\nu\in\M_{1,S}(\Omega)$ such that $\int_\Omega(I_\nu+\phi)d\nu=P(\phi)$ \cite{s99}*{Corollary 2}. 
Thus, if we slightly broaden the notion of an equilibrium state to include any measure that realises the variational equation \eqref{eq: VP in incorporated form for cms}, rather than requiring the separated form \eqref{eq: VP in separated form for cms}, we conclude that every Hölder potential $\phi$ with finite pressure on a topologically mixing countable Markov shift $\Omega$ with the BIP property admits a unique equilibrium state in this generalised sense.
Thus, in particular, in Example \ref{the folklore example in cms}, the potential $\varphi$ has an equilibrium state in our generalised sense, which is $\lambda$, since $I_\lambda+\varphi$ is $\lambda$-integrable and $\int_\Omega (I_\lambda+\varphi)d\lambda=0=P(\varphi)$.
However, as $h(\lambda)+\int_\Omega\phi d\lambda$ is not defined, $\varphi$ does not have an equilibrium state in the sense of \eqref{eq: VP in separated form for cms}.

\begin{rmk}
    The validity of the variational equation \eqref{eq: VP in separated form for cms} and the existence of an equilibrium state for a potential $\phi$ is a very subtle issue for the countable Markov shifts, because of the absence of compactness of the underlying space.
    However, if $\Omega\Subset E^{\Z_+}$ is a compact subshift, then all these subtleties disappear.
    Namely, for any continuous potential $\phi\in C(\Omega)$, the pressure \eqref{eq: gurevich pressure} is well-defined and finite and
    \begin{equation}\label{eq: vp for compact subshifts}
        P(\phi)=\sup\Big \{h(\nu)+\int_\Omega\phi d\nu: \nu\in\M_{1,S}(\Omega)\Big\}.
    \end{equation}
    Furthermore, there exists an equilibrium state $\mu$ for $\phi$, though may not be unique. 
\end{rmk}

\subsection{Wasserstein and \texorpdfstring{$\bar d$}{d-bar} distances}
Let $(X,d)$ be a bounded separable metric space, and consider probability measures $\mu,\nu$ on $X$. 
A coupling $\Pi$ of $\mu$ and $\nu$ is a probability measure on $X\times X$ whose marginals are $\mu$ and $\nu$; that is, for every Borel set $B$, $\Pi(B\times X)=\mu(B)$ and $\Pi(X\times B)=\nu(B)$.
We denote the set of all couplings of $\mu$ and $\nu$ by $\cpl(\mu,\nu).$
Then the Wasserstein-1 distance between $\mu$ and $\nu$ is defined by
\begin{equation}
    W_{1,d}(\mu,\nu):=\inf\Big\{\int_{X\times X} d(x,y) \Pi(dx,dy): \Pi\in\cpl(\mu,\nu) \Big\}.
\end{equation}
By the Kantorovich-Rubinstein theorem \cite{Dudley-book}*{Theorem 11.8.2}, $W_{1,d}$ can also be given as
\begin{equation}
    W_{1,d}(\mu,\nu)
    =
    \sup\Big\{ \int_X f d\mu- \int_X f d\nu: f\in C(X),\; \sup_{x\neq y} \frac{|f(x)-f(y)|}{d(x,y)}\leq 1\Big\}.
\end{equation}
For shift-invariant probability measures $\mu,\nu\in\M_{1,S}(\Omega)$, a \textit{joining} of $\mu$ and $\nu$ is a coupling $\Pi\in\M_{1}(\Omega\times\Omega)$, which is also shift-invariant, i.e., for every $B_1,B_2\in\F$, $\Pi(B_1\times B_2)=\Pi(S^{-1}(B_1)\times S^{-1}(B_2))$.
We denote the set of all joinings between shift-invariant measures $\mu$ and $\nu$ by $\cpl_S(\mu,\nu)$.

Between two finite words $\xi_0^{n-1}, \eta_0^{n-1}\in\Lt_n$, we define the Hamming distance $d^{(n)}(\xi_0^{n-1}, \eta_0^{n-1})$ by
\begin{equation*}
    d^{(n)}(\xi_0^{n-1}, \eta_0^{n-1})
    :=
    \frac{1}{n}\sum_{i=0}^{n-1}\rho(\xi_i,\eta_i),
\end{equation*}
where $\rho$ denotes the discrete metric on $E$, i.e., for $a,b\in E$, $\rho(a,b)=1$ if $a\neq b$ and $0$ otherwise.
For probability measures, $\nu,\mu\in\M_{1,S}(\Omega)$, let's denote their marginals $\nu|_{\F_{[0,n)}}$, $\mu|_{\F_{[0,n)}}$ by $\nu_n$ and $\mu_n$ respectively, and we denote the restriction of $\Omega$ to the volume $[0,n-1]$, which can be identified with $\Lt_n$, by $\Omega_n$.
Then the Wasserstein-1 distance between $\mu_n$ and $\nu_n$ is
\begin{equation}
    W_{1,d^{(n)}}(\mu_n,\nu_n)=\inf\Big \{\int_{\Omega_n\times\Omega_n} d^{(n)}(\xi_0^{n-1}, \eta_0^{n-1})\Pi_n(d\xi_0^{n-1}, d\eta_0^{n-1}): \Pi_n\in \cpl(\mu_n,\nu_n) \Big\}.
\end{equation}
Then \cite{GNS75}*{Theorem 1} yields the following proposition.
\begin{prop}\label{prop: bar d is defined in infinite volume}
    For all shift-invariant measures $\mu,\nu\in\M_{1,S}(\Omega)$, the finite-volume distances 
    $
    W_{1,d^{(n)}}(\mu_n,\nu_n)
    $
    converge to $\bar d(\mu,\nu):=\sup_{n} W_{1,d^{(n)}}(\mu_n,\nu_n)$ as $n\to\infty$.
    Furthermore, $\bar d(\mu,\nu)$ equals to 
    \begin{equation}\label{eq: equiv of Hamming distance and joining def of bar d}
        \inf\Big\{\int_{\Omega\times\Omega} \rho(\xi_0, \eta_0) \Pi(d\xi,d\eta): \Pi\in \cpl_S(\mu,\nu) \Big\}.
    \end{equation}
\end{prop}
$\bar d$ in Proposition \ref{prop: bar d is defined in infinite volume} becomes a metric on the space of shift-invariant probability measures $\M_{1,S}(\Omega)$ \cite{GNS75}*{page 317}, and the topology generated by $\bar d$ is finer than the weak$^*$ topology on $\M_{1,S}(\Omega)$ (see \cite{Shields-book}*{Section I.9.b.4} and also Proposition \ref{prop: Was dist vs bar d dist} below).

\subsection{Gaussian concentration bound}
A function $K:\Omega^n\to\R$ is to be called \textit{separately $d_\theta$-Lipschitz} if, for all $i$, there exists a constant $\lip_{\theta, i}(K)>0$ with
\begin{equation}
    \lvert K(x_0,\dots, x_i, \dots, x_{n-1}) - K(x_0, \dots, x_i',\dots, x_{n-1}) \rvert 
    \leq
    \lip_{\theta, i}(K)d_\theta(x_i,x_i'),
\end{equation}
for all $x_0,\dots,x_i,\dots, x_{n-1}, x_i'\in\Omega$.
We say a probability measure $\mu\in \M_1(\Omega)$ satisfies the \textit{Gaussian concentration bound} with a constant $D$ if for every $n\in\N$ and for every separately $d_\theta$-Lipschitz function $K:\Omega^n\to\R$,
\begin{equation}\label{eq: global GCB def}
    \int_\Omega e^{K(x, Sx,\dots, S^{n-1}x) } \mu(dx)
    \leq
    \exp{\Big(\int_\Omega K(x, Sx,\dots, S^{n-1}x) \mu(dx)\Big) }
    e^{D\sum_{i=0}^{n-1}\lip_{\theta,i}(K)^2}.
\end{equation}

In \cite{CMU19}, in the setup of finite state-space, authors established the Gaussian concentration bound for the equilibrium state of potentials with a subexponential continuity rate.
Later, in \cite{MMNL21}, authors extended this work by establishing Gaussian concentration bounds for the countable Markov shifts.

\begin{thm}\cite{MMNL21}*{Theorem 3.1}\label{eq: GCB for Markov shifts}
Let $(\Omega,S)$ be a topologically mixing countable Markov shift satisfying the BIP property.
Assume $\phi:\Omega\to\R$ is a H\"older potential with finite pressure, and $\mu$ is its associated equilibrium state (in the generalised sense).
Then $\mu$ satisfies the Gaussian concentration bound \eqref{eq: global GCB def} with some constant $D$.

\end{thm}

\section{The transport-entropy inequality implies the EKP inequality}\label{section: TEI implies EKP}

\subsection{The pressure defect as the relative entropy density}
Recall that \textit{the relative entropy} between finite volume marginals $\nu_n=\nu|_{\F_{[0,n)}}, \mu_n=\mu|_{\F_{[0,n)}}$ of two probability measures $\mu,\nu\in\M_{1}(\Omega)$ is defined by
\begin{equation}
    H(\nu_n|\mu_n):=\int_\Omega \log \frac{d\nu_n}{d\mu_n} d\nu_n
    =
    \sum_{a_0^{n-1}\in\Lt_n} \nu([a_0^{n-1}]) \log \frac{\nu([a_0^{n-1}])}{\mu([a_0^{n-1}])}
    \geq 
    0.
\end{equation}
For shift-invariant probability measures $\nu,\mu\in\M_{1,S}(\Omega)$, the \textbf{\textit{relative entropy density}} $h(\nu|\mu)$ is defined by the limit
\begin{equation}\label{def of rel entr dens}
    h(\nu|\mu):=\lim_{n\to\infty} \frac{1}{n}H(\nu_n|\mu_n),
\end{equation}
provided the limit exists, and it is allowed to be $+\infty$.
The existence of the limit in \eqref{def of rel entr dens} is not automatic even in the finite state-space setting.
However, the existence of relative entropy density is guaranteed under some mild conditions on $\mu$, such as Pfister's asymptotic decoupling \cite{Pfi02}*{Proposition 3.2}.
The (DLR) Gibbs measures, or their renormalisation group transformations, enjoy the asymptotic decoupling property \cite{Pfi02}.
Another rather dynamical condition, applies when $(\Omega,S)$ is compact and guarantees the existence of relative entropy density, is the \textit{weak Bowen-Gibbs property}.
We call a shift-invariant probability measure $\mu\in\M_{1,S}(\Omega)$ to have the weak Bowen-Gibbs property with respect to a potential $\phi\in C(\Omega)$
if there exist constants $P\in\R$ and $C_n\geq 1, \; n\in\N$ such that $\lim_{n\to\infty}\frac{1}{n}\log C_n=0$ and for every $n\in\N$ and $\omega\in\Omega$,
\begin{equation}\label{weak BG}
    \frac{1}{C_n}\leq \frac{\mu([\omega_0^{n-1}])}{e^{S_n\phi(\omega)-nP}}
    \leq
    C_n.
\end{equation}
When the state space $E$ is finite, the weak Bowen-Gibbs property also implies that $P$ is the pressure and $\mu$ is an equilibrium state of the potential $\phi$, as well as it implies the following relation between the relative entropy density and the pressure defect: for every $\nu\in\M_{1,S}(\Omega)$,
\begin{equation}\label{eq: rel ent dens=pressure gap for sfts}
    h(\nu|\mu)
    =
    P-h(\nu)-\int_\Omega\phi d\nu.
\end{equation}
In the literature, the weak Bowen-Gibbs property, as well as the relation \eqref{eq: rel ent dens=pressure gap for sfts}, have been established in various setups: see \cite{Georgii-book} for the account in statistical mechanics and \cite{MV2026} for the account in ergodic theory.

The measures considered in this paper have a property slightly stronger than \eqref{weak BG}.
A shift-invariant probability measure $\mu$ on the  shift space $\Omega$ is called \textit{Bowen-Gibbs measure} for a potential $\phi\in C(\Omega)$, if there exist constants $P\in\R$ and $C\geq 1$ such that for every $n\in\N$ and $\omega\in\Omega$,
\begin{equation}\label{eq: BG}
    \frac{1}{C}\leq \frac{\mu([\omega_0^{n-1}])}{e^{S_n\phi(\omega)-nP}}
    \leq
    C.
\end{equation}

The following proposition establishes the existence of the relative entropy density as well as the relation \eqref{eq: rel ent dens=pressure gap for sfts} in the setting of countable Markov shifts.

\begin{prop}\label{prop: rel ent den=press gap}
    Let $(\Omega,S)$ be a topologically mixing countable Markov shift.
    Assume a shift-invariant probability measure $\mu\in\M_{1,S}(\Omega)$ satisfies the Bowen-Gibbs property \eqref{eq: BG} with respect to a
    H\"older
    potential $\phi$.
    Then for any shift-invariant probability measure $\nu\in\M_{1,S}(\Omega)$,
    the relative entropy density $h(\nu|\mu)$ exists and 
    \begin{equation}\label{eq: rel entropy dens=pressure gap in gener sense}
    h(\nu|\mu)=P-\int_\Omega (I_\nu+\phi) d\nu.
    \end{equation}
In particular, if $(h(\nu), \int_{\Omega}\min\{\phi, 0\}d\nu)\neq (+\infty,-\infty)$, then
    
    \begin{equation}\label{eq: rel entropy density=pressure gap}
    h(\nu|\mu)
    =
    P-h(\nu)-\int_\Omega \phi d\nu.
\end{equation}
\end{prop}

\begin{proof}
Under the conditions of the above proposition, it follows from \cite{s03}*{Theorem 1} that the Markov shift $\Omega$ has the BIP property and the pressure of $\phi$ is finite and $P=P(\phi)$. 

Then by the Bowen-Gibbs property \eqref{eq: BG}, for all  $n\in\N$ and $\omega\in\Omega$,
\begin{equation}
    \Big \lvert\log \frac{\nu([\omega_0^{n-1}])}{\mu([\omega_0^{n-1}])}
    -
    nP+G_n(\omega)\Big\rvert
    \leq
    \log C,
\end{equation}
where $G_n(\omega)=-\log \nu([\omega_0^{n-1}]) + S_n\phi(\omega)$ which is finite for $\nu$- a.e. $\omega\in\Omega$.
Hence,
\begin{equation}\label{eq: H_n G_n 1st rel}
    H(\nu_n|\mu_n)
    =
    nP-\int_\Omega G_n d\nu
    +O(1).
\end{equation}
In order to write $G_n$ as an ergodic sum, for $0\leq k\leq \infty$, set
\begin{equation*}
    f_k(\omega)
    :=
    -\log\nu(\omega_0|\omega_1^k)
    +\phi(\omega),
\end{equation*}
which is well-defined and finite for $\nu$-a.e. $\omega\in\Omega$.
Note that $f_0(\omega)=-\log\nu([\omega_0])+\phi(\omega)$, and for the conditional probability, $\nu(\omega_0|\omega_1^k)
=
\frac{\nu([\omega_0^k])}{\nu([\omega_1^k])}$.
Then, using the telescoping,
\begin{equation*}\label{eq: telescoping}
    \nu([\omega_0^{n-1}])
    =
    \prod_{i=0}^{n-1}\frac{\nu([\omega_i^{n-1}])}{\nu([\omega_{i+1}^{n-1}])}
    =
    \prod_{i=0}^{n-1} \nu(\omega_i|\omega_{i+1}^{n-1}),
\end{equation*}
we obtain that
\begin{equation}\label{consequence of telescoping}
    G_n(\omega)
    =
    \sum_{i=0}^{n-1}
    f_{n-1-i}\circ S^{i}(\omega).
\end{equation}
By the martingale convergence theorem, we also have that
\begin{equation}
f_k(\omega)\xrightarrow[k\to\infty]{} f_\infty(\omega)
=
I_\nu(\omega)+\phi(\omega),
\text{ for } \nu \text{-a.e. } \omega\in\Omega.
\end{equation}
Now, we show that for all $k\in\Z_+\cup \{\infty\}$, $f_k$ is $\nu$-integrable, i.e., the integral $\int_\Omega f_k d\nu$ of $f_k$ exists as an extended real number.

We first note that $\phi$ is bounded above.
In fact, the Bowen-Gibbs condition \eqref{eq: BG} with $n=1$ yields that for every $a\in \Lt_1$, $q_a:=e^{\sup_{[a_0]}\phi}\leq Ce^{P}\mu{([a_0])}$, in particular, $\sum_{a\in \Lt_1} q_a \leq Ce^P$.
Furthermore, by the definition of $q_a$, 
$
    \phi(\omega)\leq \log q_{\omega_0},
$
for every $\omega\in\Omega$,
which yields that
\begin{equation}\label{eq: 1st bound on f_k}
    f_k(\omega)\leq -\log \nu(\omega_0|\omega_1^k)
    +
    \log q_{\omega_0}, \;\; \omega\in\Omega.
\end{equation}
Fix $a\in \Lt_1\subset E$ and $t>0$, and consider the random variable
$M_k^a(\omega):=\nu([a_0]|\omega_1^k)$, and define a hitting time by
\begin{equation*}
    \st_{a,t}(\omega):=\inf\{k\in\Z_+: M^a_k(\omega)< q_ae^{-t}\}, \;\; \omega\in\Omega.
\end{equation*}
Then since $\{\omega: \st_{a,t}=k\} \in \F_{[1,k]}$ for every $k\geq 0$,
$$
\nu\Big (X_0=a, \st_{a,t}=k|\F_{[1,k]}\Big)
=
\mathds{1}_{\{\st_{a,t}=k\}} \cdot
\nu\Big (X_0=a|\F_{[1,k]}\Big)
=
\mathds{1}_{\{\st_{a,t}=k\}} 
M_k^a,
$$
hence
\begin{equation*}
\nu\Big (X_0=a, \st_{a,t}=k\Big)
=
\int_\Omega\mathds{1}_{\{\st_{a,t}=k\}} 
M_k^a \; d\nu
\leq
q_a e^{-t}\cdot
\nu(\{\st_{a,t}=k\}).
\end{equation*}
Consequently,
\begin{equation}\label{eq: bound for nu(a_0, st)}
\nu\Big (X_0=a, \; \st_{a,t}<\infty\Big)
\leq
q_a e^{-t}\cdot
\sum_k\nu(\{ \st_{a,t}=k\})
\leq
q_a e^{-t}.
\end{equation}
Now consider a configuration $a_0\omega_1^{\infty}$ in $[a_0]\cap \{\sup_k f_k>t\}$, then by \eqref{eq: 1st bound on f_k}, for a large $k$, 
\begin{equation*}
    t<f_k(a_0\omega_1^{\infty})
    \leq
    -\log M_k^a(\omega) + \log q_a
    \;\;\; \Longrightarrow 
    \;\;\;
    M_k^a(\omega)< q_a e^{-t}
    \;\;\; \Longrightarrow 
    \;\;\;
    \st_{a,t}<\infty.
\end{equation*}
By \eqref{eq: bound for nu(a_0, st)}, $\nu([a_0]\cap \{\sup_k f_k>t\})\leq \nu\Big (X_0=a, \; \st_{a,t}<\infty\Big)
\leq
q_a e^{-t}$, which implies
\begin{equation}\label{eq: bound before showing integr of f_k}
    \nu\Big (\{\sup_k f_k>t\}\Big)
    =
    \sum_{a}
    \nu([a_0]\cap \{\sup_k f_k>t\})
    \leq
    e^{-t}\sum_aq_a.
\end{equation}
Hence as $(\sup_k f_k)^+=\sup_k f_k^+$, \eqref{eq: bound before showing integr of f_k} yields
\begin{equation}\label{eq: bound for integrals of f_k^+}
    \int_\Omega \sup_k f_k^+\;  d\nu
    \leq
    \sum_aq_a \int_0^\infty e^{-t} dt
    =
    \sum_aq_a
    \leq
    C e^P.
\end{equation}
Since $f_k$ converges to $f_\infty$ $\nu$-a.s. as $k\to\infty$,
we also conclude 
$
\int_\Omega f_\infty^+ d\nu 
    \leq
    C e^P.
$
\\
For $k\in\Z_+\cup \{\infty\}$, set $r_k:=\int_\Omega f_k d\nu\in [-\infty, \infty)$.
Note that as $\nu$ is shift-invariant, by inserting \eqref{consequence of telescoping} into \eqref{eq: H_n G_n 1st rel}, we have
\begin{equation}
    H(\nu_n|\mu_n)
    =
    nP-\sum_{k=0}^{n-1} r_k
    +O(1).
\end{equation}
Therefore, if we can show that $r_k$ converges to $r_\infty=\int_\Omega f_\infty d\nu=\int_\Omega(I_\nu+\phi)d\nu$ as $k\to\infty$, then the Stolz-Cesàro theorem yields the claim of the proposition, i.e.,
\begin{equation*}
    \lim_{n\to\infty} \frac{1}{n}H(\nu_n|\mu_n)
    =
    P-\lim_{n\to\infty}\frac{1}{n}\sum_{k=0}^{n-1} r_k
    =
    P-r_\infty.
\end{equation*}
To prove the convergence of the sequence, we first show that
$\limsup_{k\to\infty} r_k\leq r_\infty$.
Note that for every $k\in\Z_+$, $f_k\leq F:=\sup_k f_k^+$, and by \eqref{eq: bound before showing integr of f_k}, $F\in L^1(\Omega,\nu)$. 
Then as $\lim_{k\to\infty }f_k=f_\infty$ $\nu$-a.e., by employing the reverse Fatou theorem, we conclude $\limsup_{k\to\infty} r_k\leq r_\infty$. 

\noindent
Next, we show for every $k\in\Z_+$ that $r_k\geq r_\infty$. 
By the monotone convergence theorem,
\begin{equation}\label{eq: decompos of ratio of cond prob}
    \int_\Omega \frac{\nu(\omega_0|\omega_1^k)}{\nu(\omega_0|\omega_1^{\infty})}\nu(d\omega)
    =
    \sum_{a\in\Lt_1}
    \int_\Omega 
    \mathds{1}_{[a_0]}(\omega)\frac{\nu(a_0|\omega_1^k)}{\nu(a_0|\omega_1^{\infty})}\nu(d\omega).
\end{equation}
Since $\omega\mapsto\frac{M_k^a(\omega)}{M_\infty^a(\omega)}= \frac{\nu(a_0|\omega_1^k)}{\nu(a_0|\omega_1^{\infty})}$ 
is $\F_{[1,\infty)}$ measurable, we have
$$
\nu\Big(\mathds{1}_{[a_0]}\frac{M_k^a}{M_\infty^a}\Big | \F_{[1,\infty)} \Big)
=
\frac{M_k^a}{M_\infty^a}\cdot 
\nu\Big(\mathds{1}_{[a_0]}\Big | \F_{[1,\infty)} \Big)
=
M_k^a, \;\; \text{ for } \nu-\text{a.e.}.
$$
Hence, by \eqref{eq: decompos of ratio of cond prob},
\begin{equation}\label{eq: integral of ratio of cond probab}
    \int_\Omega \frac{\nu(\omega_0|\omega_1^k)}{\nu(\omega_0|\omega_1^{\infty})}\nu(d\omega)
    =
    \sum_{a\in\Lt_1}
    \int_\Omega 
    \nu(a_0|\omega_1^k)\nu(d\omega)
    =
    \int_\Omega 
    \sum_{a\in\Lt_1}
    \nu(a_0|\omega_1^k)\nu(d\omega)
    =1.
\end{equation}
Consider $\ell_k(\omega):=f_k(\omega)-f_\infty(\omega)=\log \frac{\nu(\omega_0|\omega_1^\infty)}{\nu(\omega_0| \omega_1^k)}$ defined for $\nu$-a.e. $\omega$.
Then by the elementary inequality $1-\frac{1}{s}\leq \log s$, which holds for all $s\in [0,\infty]$, it follows that for $\nu$-almost all $\omega$, $\ell_k(\omega)\geq \mathfrak{u}_k(\omega):= 1-\frac{\nu(\omega_0|\omega_1^k)}{\nu(\omega_0| \omega_1^\infty)}$.
Hence for the negative part $\ell_k^-$ of $\ell_k$, one has 
$
\ell^-_k(\omega)\leq \mathfrak{u}_k^-(\omega)
=
\Big( \frac{\nu(\omega_0|\omega_1^k)}{\nu(\omega_0| \omega_1^\infty)}-1\Big)^+
\leq
\frac{\nu(\omega_0|\omega_1^k)}{\nu(\omega_0| \omega_1^\infty)}.
$
Thus by \eqref{eq: integral of ratio of cond probab}, $\ell_k^-\in L^1(\Omega,\nu)$, which implies that $\ell_k$ is integrable.
Furthermore, for the integral of $\ell_k$, \eqref{eq: integral of ratio of cond probab} also implies that 
$
\int_\Omega \ell_k d\nu
\geq
\int_\Omega \mathfrak{u}_k d\nu
=
0$.
By the definition of $\ell_k$, $f_\infty=f_k-\ell_k$, and since $f_k^+, (-\ell_k)^+\in L^1(\Omega,\nu)$, the integral $\int_\Omega (f_k-\ell_k)d\nu$ can be split.
Therefore, by integrating, we get that $r_\infty=r_k-\int_\Omega \ell_k d\nu$, thus since $\int_\Omega \ell_k d\nu\geq 0$, we conclude the desired claim.

Note that if $(h(\nu), \int_{\Omega}\min\{\phi, 0\}d\nu)\neq (+\infty,-\infty)$, then the two sums in the expression 
$$
H(\nu_n|\mu_n)
=
\sum_{a_0^{n-1}\in\Lt_n} \nu([a_0^{n-1}])\log\nu ([a_0^{n-1}])
-
\sum_{a_0^{n-1}\in\Lt_n} \nu([a_0^{n-1}])\log\mu ([a_0^{n-1}])
$$
can be treated separately, and then \eqref{eq: rel entropy density=pressure gap} can be established using the Bowen-Gibbs property \eqref{eq: BG}.
\end{proof}

\subsection{A quantitative relation between the Wasserstein distance \texorpdfstring{$W_{1,d_\theta}$}{W1,dtheta} and \texorpdfstring{$\bar d$}{bard}}
In this subsection, we prove the following inequality.
\begin{prop}\label{prop: Was dist vs bar d dist}
    For all shift-invariant probability measures $\nu,\mu\in\M_{1,S}(\Omega)$, one has 
    \begin{equation}\label{Was dist vs bar d dist}
    W_{1,d_\theta}(\nu,\mu)\leq \frac{1}{1-\theta} \bar d(\nu,\mu).
\end{equation}
Here the constant $\frac{1}{1-\theta}$ is optimal. 
\end{prop}
\begin{rmk}
    The reverse of the relationship \eqref{Was dist vs bar d dist} can not be true because $\bar d$ generates a topology strictly finer than the weak$^*$ topology.
\end{rmk}

\begin{proof}
Let us introduce another metric on $\Omega$ by
$
d_T(\xi,\eta):=\sum_{i=0}^{\infty} \theta^{i+1}\rho(\xi_i,\eta_i)
$ 
$\xi, \eta\in\Omega$.
Then for every $\xi, \eta\in\Omega$, it is easy to check that
\begin{equation}\label{eq: d_theta vs d_T}
    \theta d_{\theta}(\xi,\eta)
    \leq
    d_T(\xi,\eta)
    \leq
    \frac{\theta}{1-\theta} d_\theta(\xi,\eta).
\end{equation}
Fix shift-invariant probability measures $\mu,\nu\in\M_{1,S}(\Omega)$, and consider a joining $\Pi\in\cpl_S(\mu,\nu)$.
Then
\begin{eqnarray}\label{eq: a tr-inv trick in metric comparison}
    \int_{\Omega\times\Omega} d_T(\xi,\eta)\Pi(d\xi,d\eta)
    &=&\notag
    \sum_{i=0}^{\infty}\theta^{i+1} \int_{\Omega\times\Omega} \rho(\xi_i,\eta_i) \Pi(d\xi,d\eta)\\
    &=&\notag
    \sum_{i=0}^{\infty}\theta^{i+1} \int_{\Omega\times\Omega} \rho(\xi_0,\eta_0) \Pi(d\xi,d\eta)\\
    &=&
    \frac{\theta}{1-\theta}
    \int_{\Omega\times\Omega} \rho(\xi_0,\eta_0) \Pi(d\xi,d\eta)
\end{eqnarray}
in the second equation above, we used the invariance of the coupling $\Pi$ under the map $S\times S$.
Then as $\cpl_S(\mu,\nu)\subset \cpl(\mu,\nu)$, we can conclude from Proposition \ref{prop: bar d is defined in infinite volume} and \eqref{eq: a tr-inv trick in metric comparison} that
\begin{multline*}
    W_{1,d_T}(\mu,\nu)
    =
    \inf_{\Pi\in\cpl(\mu,\nu)}\int_{\Omega\times\Omega} d_T(\xi,\eta)\Pi(d\xi,d\eta)
    \\ \label{eq: W_{1,d_T} Vs bar d}
    \leq
    \frac{\theta}{1-\theta}
    \inf_{\Pi\in\cpl_S(\mu,\nu)}
    \int_{\Omega\times\Omega} \rho(\xi_0,\eta_0) \Pi(d\xi,d\eta)
    =
    \frac{\theta}{1-\theta} \bar d(\mu,\nu).
\end{multline*}
Thus, by \eqref{eq: d_theta vs d_T}, 
the statement of the proposition follows.

For optimality of the multiplicative constant, consider $\Omega:=\{0,1\}^{\Z_+}$ and let $\mu=\delta_{\bar 0}$, and $\nu_p=(p\delta_0+(1-p)\delta_{1})^{\Z_+}$ for $p\in(0,1)$, where $\bar 0$ is all $0$ configuration.
Then since $\cpl(\delta_{\bar 0},\nu_p)=\{\delta_{\bar 0}\times\nu_p\}$, one has $W_{1,d_\theta}(\delta_{\bar 0},\nu_p)=\frac{1-p}{1-\theta p}$ and $\bar d(\delta_{\bar 0},\nu_p)=1-p$, which yields 
$\sup_{p\in(0,1)}\frac{W_{1,d_\theta}(\delta_{\bar 0},\nu_p)}{\bar d(\delta_{\bar 0},\nu_p)}=\frac{1}{1-\theta}$.

\end{proof}

\begin{cor}
In light of \eqref{Was dist vs bar d dist}, provided  \eqref{eq: rel entropy dens=pressure gap in gener sense}, which is true in the setups where either the transport-entropy inequality or the effective intrinsic ergodicity is established,  the transport-entropy inequality \eqref{eq: tei} immediately yields the EKP inequality \eqref{eq: ekp in W_{1,d_theta} norm} with the constant $\frac{1}{1-\theta} C_{\tei}$, i.e., \begin{equation}
    W_{1,d_\theta}(\nu,\mu_\phi)
    \leq
    \frac{C_{\tei}}{1-\theta} \sqrt{P(\phi)-h(\nu)-\int_\Omega \phi d\nu}, \;\; \; \nu\in\M_{1,S}(\Omega).
\end{equation} 
\end{cor}

\section{Transport-entropy inequality for countable Markov shifts and its applications}\label{Section: TEI on CMS}
\subsection{Transport-entropy inequality for countable Markov shifts}
In this subsection, we aim to extend the transport-entropy inequality proved in \cite{CMU19}*{Theorem 3.9} to the setup of countable Markov shifts.
We in fact show the following.
\begin{thm}\label{thm: my tei on CMS}
Suppose $(\Omega,S)$ is a topologically mixing countable Markov shift satisfying the BIP property.
Let $\phi:\Omega\to\R$ be a H\"older potential with finite pressure, and $\mu$ be its unique Bowen-Gibbs measure.
Then there exists a constant $C_{\tei}=C_\tei(\phi)>0$ such that for all $n\geq 1$ and for every shift-invariant measure $\nu\in\M_{1,S}(\Omega)$, one has
\begin{equation}\label{eq: fin volume tei}
    W_{1,d^{(n)}}(\nu_{n},\mu_n)
    \leq
    C_{\tei} \sqrt{\frac{1}{n}H(\nu_n|\mu_n)}.
\end{equation}
In particular,
\begin{equation}\label{eq: tei in infty volume}
    \bar d(\nu,\mu)
    \leq
    C_{\tei} \sqrt{h(\nu|\mu)}.
\end{equation}

\end{thm}

\begin{proof}
The theorem follows from the Gaussian concentration bound in Theorem \ref{eq: GCB for Markov shifts} and the Bobkov-G\ou tze theorem (Theorem 3.1 in \cite{BG99}).
In fact, Theorem 3.1 in \cite{BG99} states that the finite volume transport-entropy inequality \eqref{eq: fin volume tei} is equivalent to the finite volume GCB; namely, the inequality
\begin{equation}
    W_{1,d^{(n)}}(\nu_n,\mu_n)\leq \sqrt{2C H(\nu_n|\mu_n)}
\end{equation}
holds for any probability measure $\nu_n\in\M_1(\Omega_n)$, which is absolutely continuous with respect to $\mu_n$, if and only if, for every $d^{(n)}$-Lipschitz function $f:\Omega_n\to\R$ and for all $t\in\R$,
\begin{equation}
    \int_{\Omega_n}e^{tf}d\mu_n
    \leq
    \exp{\Big(\int_{\Omega_n}tf\; d\mu_n\Big)}
    \exp{\Big(\frac{{C|f|_{d^{(n)}}^2t^2}}{2}\Big)},
\end{equation}
where $|f|_{d^{(n)}}$ is the Lipschitz constant of $f$.

Fix Lipschitz function $f:\Omega_n\to\R$ with Lipschitz constant $L=|f|_{d^{(n)}}$. 
Consider an extension $\tilde{f}:E^n\to\R$ by $\tilde f(\sigma_0^{n-1}):=\inf_{\xi_0^{n-1}\in\Omega_n}\{f(\xi_0^{n-1})+L\; d^{(n)}(\sigma_0^{n-1}, \xi_0^{n-1})\}$, where the Hamming distance $d^{(n)}$ is given on $E^n$.
One can easily check that $\tilde f|_{\Omega_n}=f$ and $\tilde f$ is also Lipschitz with constant $|\tilde f|_{d^{(n)}}=L$, thus $|\tilde f|_{d^{(n)}}=|f|_{d^{(n)}}$.
Extend $\tilde f$ to $\Omega^n$ as follows
$$
F(\omega^{(0)},\dots,\omega^{(n-1)})
:=
\tilde f(\omega^{(0)}_0,\dots,\omega^{(n-1)}_0),\;\; (\omega^{(0)},\dots,\omega^{(n-1)})\in\Omega^n.
$$
We define the oscillation of $\tilde f:E^n\to\R$ at a coordinate $i=\overline{0,n-1}$ by $\partial_i\tilde f:=\sup\{\tilde f(\xi_0^{n-1})-\tilde f(\eta_0^{n-1}): \xi_j=\eta_j \text{ for all } j\neq i\}$.
Then since for all $\xi_0^{n-1}, \eta_0^{n-1}\in E^n$,
$$
|\tilde f(\xi_0^{n-1})-\tilde f(\eta_0^{n-1})|
\leq
\sum_{i=0}^{n-1}\rho(\xi_i,\eta_i) \partial_i \tilde f,
$$
one can easily check that 
\begin{equation}\label{eq: rel between lips constant and oscillations}
    |\tilde f|_{d^{(n)}}=n\sup_{0\leq i\leq n-1} \partial_i \tilde f.
\end{equation}
Furthermore, it can be readily verified that for every $0\leq i\leq n-1$, $\lip_{\theta,i}(F)\leq \partial_i \tilde f.$
For any smooth function $\mathfrak{g}:\R\to\R$, the following equalities hold
\begin{eqnarray}
    \int_\Omega \mathfrak{g}\Big(F(\omega,S\omega,\dots S^{n-1}\omega)\Big)\mu(d\omega)
    &=&\notag 
    \int_\Omega \mathfrak{g}\Big(\tilde f(\omega_0,\omega_1,\dots, \omega_{n-1})\Big) \mu(d\omega)\\
    &=&\notag
    \int_\Omega \mathfrak{g}\Big(f(\omega_0,\omega_1,\dots, \omega_{n-1})\Big) \mu(d\omega)\\
    &=&\label{eq: imp extension trick}
    \int_{\Omega_n} \mathfrak{g}\circ f d\mu_n,
\end{eqnarray}
where the first and the last equalities follow from the definitions of $F$ and $\mu_n$, and the second one follows from the fact that $\tilde f|_{\Omega_n}=f$ and $\omega$ runs within $\Omega$ not in $E^{\Z_+}$.

\noindent
Since $F$ is separately $d_\theta$-Lipschitz, by applying Theorem \eqref{eq: GCB for Markov shifts} to $F$, we obtain for every $t\in\R$ that
\begin{eqnarray*}
    \int_{\Omega_n} e^{tf}d\mu_n
    &\overset{\eqref{eq: imp extension trick}}{=}&
    \int_\Omega e^{tF(\omega, S\omega,\dots, S^{n-1}\omega) } \mu(d\omega)
    \\
    &\overset{\eqref{eq: global GCB def}}{\leq}&
     \exp{\Big(\int_\Omega tF(\omega, S\omega,\dots, S^{n-1}\omega) \mu(d\omega)\Big) }
    e^{Dt^2\sum_{i=0}^{n-1}\lip_{\theta,i}(F)^2}
    \\
    &\overset{\eqref{eq: imp extension trick}}{=}&
    e^{\int_\Omega tf \; d\mu_n}
    \cdot 
    e^{Dt^2\sum_{i=0}^{n-1}\lip_{\theta,i}(F)^2}\\
    &\overset{\eqref{eq: rel between lips constant and oscillations}}{\leq}&
    e^{\int_\Omega tf \; d\mu_n}
    \cdot 
    e^{\frac{Dt^2}{n} |f|_{d^{(n)}}^2}
\end{eqnarray*}

Hence, since $\mu_n$ is a fully supported measure on $\Omega_n$ by the Bowen-Gibbs property, the Bobkov-G\ou tze theorem immediately yields that for every $\nu_n\in\M_1(\Omega_n)$,
\begin{equation}\label{eq: eq where BG thm is applied}
    W_{1,d^{(n)}}(\nu_n,\mu_n)
\leq
\sqrt{\frac{4D}{n} H(\nu_n|\mu_n)}.
\end{equation}
\eqref{eq: tei in infty volume} follows from \eqref{eq: eq where BG thm is applied} with $C_\tei=2\sqrt{D}$,  after taking a limit $n\to\infty$ in both sides of \eqref{eq: eq where BG thm is applied}.

\end{proof}

\subsection{Applications of the transport-entropy inequality}
We apply Theorem \ref{thm: my tei on CMS} to two problems: the effective uniqueness of the equilibrium states in $\bar d$ metric on the countable Markov shifts and the continuity of the map $\phi\mapsto \mu_\phi$ that sends a H\ou lder potential $\phi$ to its equilibrium (or Bowen-Gibbs) state $\mu_\phi$, which has first been raised in \cite{CQ98}.

\vspace{0.2cm}
\noindent
\textbf{\textit{\underline{The effective uniqueness of the equilibrium states in $\bar d$ metric on CMS.}}}

\vspace{0.2cm}

By combining Theorem \ref{thm: my tei on CMS} with Proposition \ref{prop: rel ent den=press gap}, one immediately obtains the following.
\begin{cor}\label{cor: EKP in bar d on CMS, recovery of Ruhr21}
    Let $\Omega$ be a topologically mixing countable Markov shift with the BIP property.
    Suppose $\phi$ is a H\ou lder potential on $\Omega$ with finite pressure, and $\mu_\phi$ is the associated Bowen-Gibbs measure.
    Then for every $\nu\in\M_{1,S}(\Omega)$,
    \begin{equation}\label{eq: EIE in bar d on CMS}
        \bar d(\nu,\mu_\phi)\leq C_{\tei} \sqrt{P(\phi)-\int_\Omega (I_\nu+\phi)d\nu}.
    \end{equation}
\end{cor}
It should be noted that \eqref{eq: EIE in bar d on CMS} also extends  \cite{CK26}*{Theorem A} to the setup of countable Markov shifts, and also strengthens the main result of \cite{Ruh21} that previously established the effective uniqueness in the Wasserstein metric.
It is straightforward that Corollary \ref{cor: EKP in bar d on CMS, recovery of Ruhr21}, in light of Proposition \ref{prop: Was dist vs bar d dist}, also recovers the main theorem of \cite{Ruh21}.

\vspace{0.2cm}

\noindent
\textbf{\underline{\textit{Continuity of the map $\phi\mapsto\mu_\phi$ in $\bar d$ metric.}}}

\vspace{0.2cm}

Recently, in \cite{B25}, J. Bhullar studied how the RPF measures
on the full countable Markov shift vary with respect to the corresponding potential in the $\bar d$ metric.
She proved that the map $\phi\mapsto \mu_\phi$ is continuous as a map from metric space $(A, d_\theta)$ to the space $(\M_{1,S}(\Omega),\bar d)$, where $A_\theta$ is the space of the weakly H\ou lder strongly positive recurrent (SPR) potentials with finite pressure, and the metric $d_\theta$ is generated from the H\ou lder norm $\lVert \cdot  \rVert_\infty+|\cdot|_\theta$ \cite{B25}*{Theorem 1.1}. 
We note that the weakly H\ou lder condition is slightly weaker than the H\ou lder condition we are considering here, as it allows the potential to have infinite first variation $\var_1\phi=\infty$, otherwise they are equivalent.
Below, by applying Theorem \ref{thm: my tei on CMS}, we obtain a relevant result.
\begin{cor}\label{assertion: Holder continuity of Coelho-Quas map}
    Assume that $\Omega$ is a topologically mixing countable Markov shift with the BIP property.
    For H\ou lder potentials $\phi,\psi$ with finite pressures $P(\phi), P(\psi)$, let $\mu_\phi$ and $\mu_\psi$ be their Bowen-Gibbs (also RPF) measures.
    Then
    \begin{equation}\label{eq: Holder continuity of Coelho-Quas map}
        \bar d(\mu_\phi,\mu_\psi)
        \leq
        \sqrt{2}\min\{C_\tei(\phi), C_\tei(\psi)\}
        \;
        \lVert \phi-\psi\rVert_\infty^{1/2}.
    \end{equation}
    In particular, the map $\phi\mapsto \mu_\phi$ is local $\frac{1}{2}$-H\ou lder continuous as a map from the metric space $(\mathtt{H\ddot{o}l}(\Omega), \lVert \cdot \rVert_\infty )$ to the space $(\M_{1,S}(\Omega), \bar d)$, where $\mathtt{H\ddot{o}l}(\Omega)$ denotes the space of H\ou lder potentials, not necessarily bounded, with finite pressure.
\end{cor}
 Before proceeding to the proof, we make a few remarks on how the above statement relates to the earlier works \cites{CQ98,B25}.
\begin{rmk}

    \begin{itemize}
        \item[(1)] Corollary \ref{assertion: Holder continuity of Coelho-Quas map} establishes continuity of the map $\phi\mapsto\mu_\phi$ with respect to the uniform metric. 
        This distinguishes it from the continuity results of Bhullar \cite{B25} and Coelho and Quas \cite{CQ98}: the former considers the metric induced by the Hölder norm $\lVert\cdot\rVert_\theta:=\lVert\cdot\rVert_\infty+|\cdot|_\theta,$
        whereas the latter uses the norm $\lVert\cdot\rVert_{\mathtt{CQ}}:=\lVert\cdot\rVert_\infty+\sum_{n=1}^\infty\var_n(\cdot).$
        Both norms control the variations of the difference between the potentials as well as their uniform distance.
        In contrast, the uniform metric considered here imposes no restrictions on the variations.
        \item[(2)] 
        Both previous works \cites{CQ98,B25} address full shifts: Coelho and Quas \cite{CQ98} consider the finite-alphabet case, whereas Bhullar \cite{B25} treats the countable-alphabet case. In contrast, Corollary \ref{assertion: Holder continuity of Coelho-Quas map} applies to arbitrary topologically mixing countable Markov shifts with the BIP property. 
        
        \item[(3)] 
        The continuity results in \cite{CQ98}*{Theorem 3} and \cite{B25}*{Theorem 1.1} are qualitative: they establish continuity of the map $\phi\mapsto\mu_\phi$ in $\bar d$ metric without providing an explicit modulus of continuity. In contrast, estimate \eqref{eq: Holder continuity of Coelho-Quas map} gives an explicit modulus, yielding a local $\frac12$-Hölder exponent with respect to the uniform metric. 
        However, we do not expect the exponent $\frac12$ to be optimal. 
        Indeed, Bhullar \cite{B25}*{Theorem 1.3} establishes a local Lipschitz estimate for the map from $g$-functions to their corresponding $g$-measures, with respect to the uniform distance between the logarithms of the $g$-functions. This suggests that a local Lipschitz estimate may also hold for the larger class of potentials considered in Corollary \ref{assertion: Holder continuity of Coelho-Quas map}, although Bhullar's result does not directly establish such an estimate in this setting.
        
        \item[(4)] 
       Regarding the regularity assumptions on the potentials, Bhullar's result applies to weakly H\ou lder potentials with finite pressure satisfying the SPR condition, whereas Corollary \ref{assertion: Holder continuity of Coelho-Quas map} concerns H\ou lder potentials with finite pressure, without imposing the SPR condition. 
       Thus, Bhullar's result allows weaker regularity, while ours does not require SPR. The respective potential classes are therefore subject to different assumptions, and neither set of hypotheses is directly weaker than the other.

    \end{itemize}
\end{rmk}

 Below, we provide two proofs for Corollary \ref{assertion: Holder continuity of Coelho-Quas map}; one is direct and another is via the relation \eqref{eq: rel entropy dens=pressure gap in gener sense}.
 
\begin{proof}[Proof I: direct]
   We only consider the case where $\lVert\phi-\psi \rVert_\infty<\infty$, otherwise \eqref{eq: Holder continuity of Coelho-Quas map} is automatically true.
   For every $n\in\N$, consider the relative entropy $H(\mu_{\psi,n}| \mu_{\phi,n})$ between the finite-volume measures $\mu_{\psi,n}$ and $\mu_{\phi,n}$:
   \begin{multline}
       H(\mu_{\psi,n}| \mu_{\phi,n})
           =\notag
           \sum_{a_0^{n-1}\in\Lt_n}\mu_{\psi,n}([a_0^{n-1}])\log\frac{\mu_{\psi,n}([a_0^{n-1}])}{\mu_{\phi,n}([a_0^{n-1}])}\\
           =
           n (P(\phi)-P(\psi))
           +
           \sum_{a_0^{n-1}\in\Lt_n}\mu_{\psi,n}([a_0^{n-1}])
           [S_n\psi(a_0^{n-1}\omega_n^{\infty})
           -
           S_n\phi(a_0^{n-1}\omega_n^{\infty})]
           +O(1).
   \end{multline}
The second equation above follows from the Gibbs-Bowen property of the measures $\mu_\psi$ and $\mu_\phi$.
Then
\begin{equation*}
     H(\mu_{\psi,n}| \mu_{\phi,n})
     =
     n (P(\phi)-P(\psi))
     +
     \int_\Omega S_n(\psi-\phi) d\mu_{\psi}
     +O(\var_n S_n(\psi-\phi))
     +O(1).
\end{equation*}
Hence since $\var_n S_n(\psi-\phi)\leq \sum_{i=1}^n \var_i(\psi-\phi)\leq \sum_{i=1}^\infty (\var_i\psi +\var_i \phi)<\infty$, and by the shift-invariance of $\mu_\psi$, we obtain,
\begin{equation}
    H(\mu_{\psi,n}| \mu_{\phi,n})
     =
     n (P(\phi)-P(\psi))
     +
     n\int_\Omega (\psi-\phi) d\mu_{\psi}
     +O(1).
\end{equation}
After dividing both sides by $n$ and by taking limit as $n\to\infty$,
\begin{equation}\label{eq: nice formula for rel ent dens of BG meas}
    h(\mu_\psi|\mu_{\phi})
    =
    P(\phi)-P(\psi)
    +
    \int_\Omega (\psi-\phi) d\mu_{\psi}.
\end{equation}
By the elementary properties of the pressure, one has that $|P(\phi)-P(\psi)|\leq \lVert \phi-\psi\rVert_\infty$.
Hence, one concludes from \eqref{eq: nice formula for rel ent dens of BG meas} that
\begin{equation}\label{eq: rel betwen rel ent dens and sup norm of diff}
    h(\mu_\psi|\mu_{\phi})
    \leq
    2 \lVert \phi-\psi\rVert_\infty.
\end{equation}
Combination of \eqref{eq: rel betwen rel ent dens and sup norm of diff} and \eqref{eq: tei in infty volume} yields \eqref{eq: Holder continuity of Coelho-Quas map}.

\end{proof}

\begin{proof}[Proof II: via the relation \eqref{eq: rel entropy dens=pressure gap in gener sense}] We again consider $\lVert\phi-\psi\rVert_\infty<\infty$.
By \eqref{eq: rel entropy dens=pressure gap in gener sense}, 
\begin{equation}\label{eq: rel ent dens press gap rel for BG meas}
    h(\mu_\psi|\mu_\phi)
    =
    P(\phi)-\int_\Omega (I_{\mu_\psi}+\phi)d\mu_{\psi}.
\end{equation}
As $\lVert\phi-\psi\rVert_\infty<\infty$, one can split the integral $\int_\Omega [(I_{\mu_\psi}+\psi)+(\phi-\psi)]d\mu_{\psi}$.
Then since $\int_\Omega (I_{\mu_\psi}+\psi)d\mu_{\psi}=P(\psi)$, one concludes from \eqref{eq: rel ent dens press gap rel for BG meas} that 
\begin{equation*}
    h(\mu_\psi|\mu_{\phi})
    =
    P(\phi)-P(\psi)
    +
    \int_\Omega (\psi-\phi) d\mu_{\psi}.
\end{equation*}
The remainder of the argument proceeds as in Proof I.
    
\end{proof}

\section{Remarks and Discussions}\label{Section: Remarks}

We conclude with some remarks on the scope of the transport-entropy
inequality (TEI) and its relationship with other concentration
inequalities. 
We first compare the structural requirements of the
TEI and the EKP inequality and discuss the range of state
spaces and spatial structures in which transport-entropy inequalities
have been studied.
We then recall the connection between the TEI and
other concentration inequalities.

\subsection{The scope of the TEI and the EKP inequality}

The TEI can be formulated in substantially greater generality than
the EKP inequality. 
Indeed, the relative entropy density $h(\nu|\mu)$ and Ornstein's $\bar d$-metric are fundamentally probabilistic notions that can be defined directly in terms of probability measures, without reference to a potential or pressure function. 
By contrast, the EKP inequality relies on additional thermodynamic
structure to be meaningfully formulated. 
This includes the existence
and, in relevant settings, the finiteness of the pressure, together
with the appropriate one-sided integrability of the function
$I_\nu+\varphi$ appearing in the pressure-defect identity.

The broader scope of the TEI is also illustrated by results outside
the one-dimensional setting considered in this paper. 
In particular, in 2006, Wu
\cite{Wu06} established the finite-volume transport-entropy inequality \eqref{eq: fin volume tei}
for various lattice models of statistical physics under the
Dobrushin uniqueness condition. These include models with finite
state spaces, such as the Ising and Potts models, as well as models
with unbounded state spaces, including the $\varphi^4$ Euclidean
quantum field with single-site state space $\mathbb R$; see Theorem 4.3, Proposition 5.5 and Proposition 5.8 in \cite{Wu06}.

\subsection{The TEI and other concentration inequalities}

The TEI is closely related to other concentration inequalities.
As used in the proof of Theorem \ref{thm: my tei on CMS}, the theorem of Bobkov and
Götze \cite{BG99}*{Theorem 3.1} establishes the equivalence between the
finite-volume TEI and the corresponding finite-volume Gaussian
concentration bound (GCB). This equivalence provides a direct link
between transport-entropy inequalities and the theory of
concentration of measure, including Gaussian concentration bounds
and logarithmic Sobolev inequalities.

The relationship between the TEI and Gaussian concentration bound has
also been studied in infinite volume. 
Recently, Chazottes, Collet, and Redig
\cite{CCR26} investigated the equivalence between Gaussian
concentration bounds and transport-entropy inequalities in this
setting. In particular, their Theorems~6.1 and~6.2 establish an
equivalence between the TEI and weaker forms of the GCB.

\section*{Acknowledgments}

This research work was supported by the Research Council of Finland via the project “\emph{Fractals and randomness}”, grant no. 368817
and also partially supported by the Swedish
Research Council under grant no. 2021-06594 while the author was in residence at Institut Mittag-Leffler in Djursholm, Sweden, in 2026.
The author also acknowledges the hospitality of the Erwin Schrödinger International Institute for Mathematics and Physics, where part of this work was carried out.


\begin{bibdiv}
\begin{biblist}

\bib{B25}{article}{
   author={Bhullar, Jasmine},
   title={$\overline {d}$-continuity for countable state shifts},
   journal={Ergodic Theory Dynam. Systems},
   volume={46},
   date={2026},
   number={2},
   pages={466--489},
   issn={0143-3857},
   review={\MR{5012852}},
   doi={10.1017/etds.2025.10241},
}

\bib{BG99}{article}{
   author={Bobkov, S. G.},
   author={G\"otze, F.},
   title={Exponential integrability and transportation cost related to
   logarithmic Sobolev inequalities},
   journal={J. Funct. Anal.},
   volume={163},
   date={1999},
   number={1},
   pages={1--28},
   issn={0022-1236},
   review={\MR{1682772}},
   doi={10.1006/jfan.1998.3326},
}


\bib{CMU19}{article}{
   author={Chazottes, J.-R.},
   author={Moles, J.},
   author={Ugalde, E.},
   title={Gaussian concentration bound for potentials satisfying Walters
   condition with subexponential continuity rates},
   journal={Nonlinearity},
   volume={33},
   date={2020},
   number={3},
   pages={1094--1117},
   issn={0951-7715},
   review={\MR{4063959}},
   doi={10.1088/1361-6544/ab5918},
}
\bib{CCR26}{article}{
   author={Chazottes, J.-R.},
   author={Collet, P.},
   author={Redig, F.},
   title={Gaussian concentration, integral probability metrics, and coupling functionals for infinite lattice systems},
   journal={arXiv:2603.17861},
   date={2026},
}
\bib{CK26}{article}{
   author={Climenhaga, V.},
   author={Kumari, T.},
   title={Effective uniqueness of the measure of maximal entropy in Ornstein's $\bar d$-metric},
   journal={arXiv:2608.21703},
   date={2026},
}

\bib{CQ98}{article}{
   author={Coelho, Zaqueu},
   author={Quas, Anthony N.},
   title={Criteria for $\overline d$-continuity},
   journal={Trans. Amer. Math. Soc.},
   volume={350},
   date={1998},
   number={8},
   pages={3257--3268},
   issn={0002-9947},
   review={\MR{1422894}},
   doi={10.1090/S0002-9947-98-01923-0},
}

\bib{Dudley-book}{book}{
   author={Dudley, R. M.},
   title={Real analysis and probability},
   series={Cambridge Studies in Advanced Mathematics},
   volume={74},
   note={Revised reprint of the 1989 original},
   publisher={Cambridge University Press, Cambridge},
   date={2002},
   pages={x+555},
   isbn={0-521-00754-2},
   review={\MR{1932358}},
   doi={10.1017/CBO9780511755347},
}

\bib{Foll73}{article}{
   author={F\"{o}llmer, Hans},
   title={On entropy and information gain in random fields},
   journal={Z. Wahrscheinlichkeitstheorie und Verw. Gebiete},
   volume={26},
   date={1973},
   pages={207--217},
   review={\MR{348829}},
   doi={10.1007/BF00532723},
}

\bib{Georgii-book}{book}{
   author={Georgii, Hans-Otto},
   title={Gibbs measures and phase transitions},
   series={De Gruyter Studies in Mathematics},
   volume={9},
   publisher={Walter de Gruyter \& Co., Berlin},
   date={1988},
   pages={xiv+525},
   isbn={0-89925-462-4},
   review={ (2nd ed. 2011)},
   doi={10.1515/9783110850147},
}

\bib{GNS75}{article}{
   author={Gray, Robert M.},
   author={Neuhoff, David L.},
   author={Shields, Paul C.},
   title={A generalization of Ornstein's $\bar d$ distance with applications
   to information theory},
   journal={Ann. Probability},
   volume={3},
   date={1975},
   pages={315--328},
   issn={0091-1798},
   review={\MR{0368127}},
   doi={10.1214/aop/1176996402},
}

\bib{HKMM26}{article}{
   author={Haydarov, F.},
   author={Kadyrov, Sh.},
   author={Makhmudov, M.},
   author={Mamayusupov, Kh.},
   title={Effective Intrinsic Ergodicity for renewal-type potentials on S-gap shifts},
   journal={arXiv:2606.17031},
   date={2026},
}



\bib{Kad15}{article}{
   author={Kadyrov, Shirali},
   title={Effective uniqueness of Parry measure and exceptional sets in
   ergodic theory},
   journal={Monatsh. Math.},
   volume={178},
   date={2015},
   number={2},
   pages={237--249},
   issn={0026-9255},
   review={\MR{3394424}},
   doi={10.1007/s00605-014-0690-7},
}

\bib{MV2026}{article}{
   author={Makhmudov, Mirmukhsin},
   author={Verbitskiy, Evgeny},
   title={Gibbs properties of equilibrium states},
   journal={Nonlinearity},
   volume={39},
   date={2026},
   number={9},
   pages={1--14},
   doi={10.1088/1361-6544/ae9e7b},
}

\bib{Mar96}{article}{
   author={Marton, K.},
   title={Bounding $\overline d$-distance by informational divergence: a
   method to prove measure concentration},
   journal={Ann. Probab.},
   volume={24},
   date={1996},
   number={2},
   pages={857--866},
   issn={0091-1798},
   review={\MR{1404531}},
   doi={10.1214/aop/1039639365},
}

\bib{MMNL21}{article}{
   author={Maldonado, Cesar},
   author={Mu\~niz, Humberto},
   author={Nieto Loredo, Hugo},
   title={Concentration inequalities and rates of convergence of the ergodic
   theorem for countable shifts with Gibbs measures},
   journal={J. Difference Equ. Appl.},
   volume={27},
   date={2021},
   number={11},
   pages={1594--1607},
   issn={1023-6198},
   review={\MR{4343287}},
   doi={10.1080/10236198.2021.2000970},
}

\bib{Pfi02}{article}{
   author={Pfister, Charles-Edouard},
   title={Thermodynamical aspects of classical lattice systems},
   conference={
      title={In and out of equilibrium},
      address={Mambucaba},
      date={2000},
   },
   book={
      series={Progr. Probab.},
      volume={51},
      publisher={Birkh\"{a}user Boston, Boston, MA},
   },
   date={2002},
   pages={393--472},
   review={\MR{1901964}},
}

\bib{P26}{article}{
   author={Pollicott, M.},
   title={Effective intrinsic ergodicity for expanding interval maps},
   journal={arXiv:2606.12081},
   date={2026},
}

\bib{Polo-thesis}{thesis}{
  author={Polo, Fabrizio},
  title={Equidistribution in chaotic dynamical systems},
  year={2011},
  note={Thesis (Ph.D.)--The Ohio State University}
}


\bib{Ruh16}{article}{
  author={R{\"u}hr, Ren{\'e}},
  title={Effectivity of uniqueness of the maximal entropy measure on
         {$p$}-adic homogeneous spaces},
  journal={Ergodic Theory Dynam. Systems},
  volume={36},
  number={6},
  pages={1972--1988},
  year={2016}
}

\bib{Ruh21}{article}{
  author={R{\"u}hr, Ren{\'e}},
  title={Pressure inequalities for Gibbs measures of countable
         Markov shifts},
  journal={Dyn. Syst.},
  volume={36},
  number={2},
  pages={332--339},
  year={2021}
}


\bib{RS22}{article}{
   author={R\"{u}hr, Ren\'{e}},
   author={Sarig, Omri},
   title={Effective intrinsic ergodicity for countable state Markov shifts},
   journal={Israel J. Math.},
   volume={251},
   date={2022},
   number={2},
   pages={679--735},
   issn={0021-2172},
   review={\MR{4527555}},
   doi={10.1007/s11856-022-2436-x},
}

\bib{s99}{article}{
   author={Sarig, Omri M.},
   title={Thermodynamic formalism for countable Markov shifts},
   journal={Ergodic Theory Dynam. Systems},
   volume={19},
   date={1999},
   number={6},
   pages={1565--1593},
   issn={0143-3857},
   review={\MR{1738951}},
   doi={10.1017/S0143385799146820},
}

\bib{s03}{article}{
   author={Sarig, Omri},
   title={Existence of Gibbs measures for countable Markov shifts},
   journal={Proc. Amer. Math. Soc.},
   volume={131},
   date={2003},
   number={6},
   pages={1751--1758},
   issn={0002-9939},
   review={\MR{1955261}},
   doi={10.1090/S0002-9939-03-06927-2},
}

\bib{Shields-book}{book}{
   author={Shields, Paul C.},
   title={The ergodic theory of discrete sample paths},
   series={Graduate Studies in Mathematics},
   volume={13},
   publisher={American Mathematical Society, Providence, RI},
   date={1996},
   pages={xii+249},
   isbn={0-8218-0477-4},
   review={\MR{1400225}},
   doi={10.1090/gsm/013},
}

\bib{Wu06}{article}{
   author={Wu, Liming},
   title={Poincar\'{e} and transportation inequalities for Gibbs measures
   under the Dobrushin uniqueness condition},
   journal={Ann. Probab.},
   volume={34},
   date={2006},
   number={5},
   pages={1960--1989},
   issn={0091-1798},
   review={\MR{2271488}},
   doi={10.1214/009117906000000368},
}
	


\end{biblist}
\end{bibdiv}
\end{document}